\documentclass[11pt]{amsart}
\usepackage{geometry}                
\usepackage{graphicx}
\usepackage{amsmath, amsthm, amssymb}
\usepackage[all]{xy}
\newtheorem{theorem}{Theorem}[section]

\newtheorem{corollary}[theorem]{Corollary} 
\newtheorem{claim}[theorem]{Claim}

\newtheorem{proposition}[theorem]{Proposition}
\newtheorem{algorithm}[theorem]{Algorithm}
\newtheorem{definition}[theorem]{Definition}
\newtheorem{remark}[theorem]{Remark}

\newcommand{\abs}[1]{\left\vert #1 \right\vert}

\begin{document}

\author{Vivian Olsiewski Healey}
\email {vhealey@nd.edu}
\title[Recompositions of the Penrose Aperiodic Protoset]{A Family of Recompositions of the Penrose Aperiodic Protoset and Its Dynamic Properties}


\maketitle


\begin{abstract} This paper describes a recomposition of the rhombic Penrose aperiodic protoset due to Robert Ammann. We show that the three prototiles that result from the recomposition form an aperiodic protoset in their own right without adjacency rules. An interation process is defined on the space of Ammann tilings that produces a new Ammann tiling from an existing one, and it is shown that this process runs in parallel to Penrose deflation. Furthermore, by characterizing Ammann tilings based on their corresponding Penrose tilings and the location of the added vertex that defines the recomposition process, we show that this process proceeds to a limit for the local geometry.
\end{abstract}

\section{Introduction}




 
 The Penrose aperiodic tiles have been well-studied by deBruijn \cite{deB}, Lunnon and Pleasants \cite{LP}, and others. See Senechal \cite{Senechal} for a survey. This paper describes a recomposition of the rhombic Penrose aperiodic protoset defined by Robert Ammann in Gr\"unbaum and Shepard \cite{GS}, showing that the three tiles that result from the construction form an aperiodic protoset in their own right without adjacency rules. An interation process is defined on the space of Ammann tilings that runs in parallel to Penrose deflation, and it is shown that this process proceeds to a limit for the local geometry.



While there are a variety of tilings attributed to Roger Penrose, the kind relevant to this paper are those admitted by a protoset of two rhombic tiles, a thin and a thick, with dimensions dependent on the golden ratio $\phi=\frac{1+\sqrt{5}}{2}$. These two tiles can be used to tile the plane non-periodically (without translational symmetry) when assembled following specific adjacency rules (see Figure~\ref{3tilings}(a)). 


An Ammann tiling (see Figure~\ref{3tilings}(b)) may be constructed from a Penrose tiling by rhombs by a process referred to as \textit{recomposition} \cite{GS}. In this process, a single vetex $Q$ is added within a thin Penrose rhomb, and edges are drawn between it and the three nearest Penrose vertices. The geometry of the newly constructed edges are then used to create two specific new vertices and five new edges inside the thick Penrose rhomb. These new vertices and edges are then copied into every Penrose rhomb in the tiling, and the original Penrose edges are deleted.


Although this construction was mentioned in passing in \cite{GS}, it has never been thoroughly studied in the literature to the knowledge of this author. This paper shows that this recomposition process produces a tiling by three prototiles that form an aperiodic protoset in their own right, that is, every tiling admitted by this protoset is non-periodic. Furthermore, after defining Ammann tilings independently of the recomposition process, I show that each Ammann tiling has a unique corresponding Penrose tiling.

For Penrose tilings, there is a process called \textit {double composition} by which a new Penrose tiling may be constructed from a starting Penrose tiling. By deleting specific edges from the original Penrose tiling, this process produces a new Penrose tiling composed of tiles geometrically similar to the originals but scaled by the golden ratio. 


With the correspondence between Penrose tilings and Ammann tilings established, we next describe an iteration process for Ammann tilings that resembles Penrose double composition.  Although this iteration process adds edges as well as deleting some, we prove that it runs in parallel to Penrose deflation.

 
In analyzing the Ammann iteration process further, we address the questions: (1) Is the iterated tiling composed of tiles geometrically similar to those of the original Ammann tiling? (2) If not, is there a sense in which the new tiling is of the same type as the original Ammann tiling? (3) Can this process be carried out using the new tiling as the starting point? and (4) If so, does the sequence of tilings approach a limit?

Although we show that the Ammann iteration process does not yield a single limit tiling, it does yield a limit for the local geometry of the tilings (see Figure \ref{3tilings} (c)). By identifying an Ammann tiling by two parameters (a) its corresponding Penrose tiling and (b) the location of the $Q$ within a thin rhomb, we show that while (a) does not approach a limit, (b) does. This result is summarized in the following theorem.

\begin{theorem}
The map $Q\mapsto Q'$ defined by Ammann iteration  has a unique attractive fixed point along the edge of the thin Penrose rhomb which divides that edge according to the golden ratio.
\label{limit_thm}
\end{theorem}

\begin{figure}
(a)$\qquad\qquad$$\qquad\qquad$$\qquad\qquad$ (b)$\qquad\qquad$$\qquad\qquad$$\qquad\qquad$ (c)

\includegraphics[scale=.35]{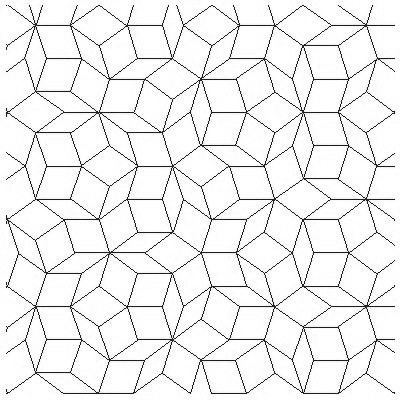}
\includegraphics[scale=.35]{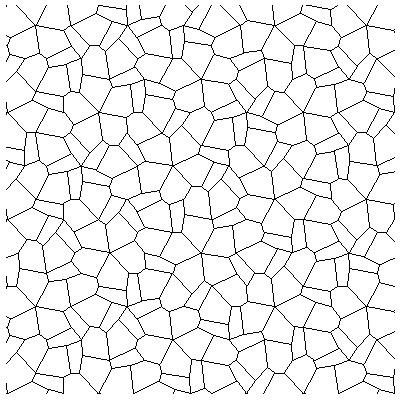}
\includegraphics[scale=.35]{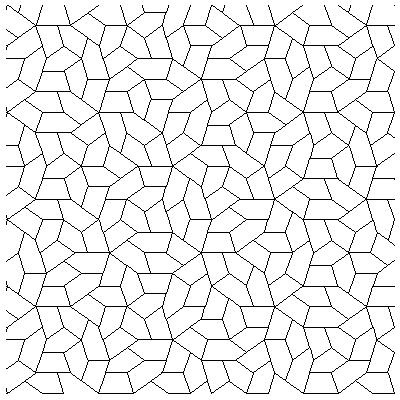}
\caption{From left to right: a patch of a Penrose tiling, the corresponding patch of a generic Ammann tiling, a patch of the Ammann limit tiling.}\label{3tilings}
\end{figure}


In section 2, we recall the necessary terminology related to tilings. In section 3, we describe Penrose tilings, including the geometry of the Penrose rhombs, the local isomorphism theorem, the definition of Penrose deflation, and the construction of identifying sequences. Section 4 describes Ammann tilings, detailing the recomposition process that constructs an Ammann tiling from a Penrose tiling, and proving that the three tiles that result from the recomposition process form an aperiodic protoset. Section 5 describes the iteration process for Ammann tilings and shows the correspondence between Ammann iteration and Penrose deflation. In section 6, we examine the dynamics of the Ammann iteration process and prove Theorem~\ref{limit_thm}. Finally, in section 7 we discuss possible application to quasicrystals.

The research for this paper was begun at the National Science Foundation sponsored Research Experience for Undergraduates at Canisius College, Summer 2008, under the guidance of Professors Terry Bisson and B.J. Kahng. The project was continued through Fall 2009 at the University of Notre Dame with Professor Arlo Caine. Research during Summer 2009 was partially funded with the help of Professor Frank Connolly through NSF Grant DMS-0601234. Many thanks to Professor Caine for the long hours he spent with me on this project and in particular for a suggestion that simplified the proof of Theorem \ref{limit}. Without his help this project would not have been possible. Also, thanks to Professor Jeffrey Diller for suggestions pertaining to the dynamics of the Ammann iteration process.

\section{Terminology}

A \textit{plane tiling} is a countable family $\mathcal{T}=\{T_{1},T_{2},...\}$ of closed subsets of the Euclidean plane, each homeomorphic to a closed circular disk, such that the union of the sets $T_{1}, T_{2},...$ (which are known as the \textit{tiles} of $\mathcal{T}$) is the whole plane, and the interiors of the sets $T_{i}$ are pairwise disjoint \cite{GS}.
We say that a set $\mathcal {S}$ of representatives of the congruence classes in $\mathcal{T}$ is a \textit{protoset} for $\mathcal{T}$, and each representative is a \textit{prototile}. If $\mathcal{S}$ is a protoset for $\mathcal{T}$, then we say that $\mathcal{S}$ \textit{admits }$\mathcal{T}$.

A \textit{patch} is a finite set of tiles whose union is simply connected.
A patch is called \textit{locally legal} if the tiles are assembled according to the relevant adjacency rules and is called \textit{globally legal} if it can be extended to an infinite tiling. 

The \textit{(first) corona} of a tile $T_{i}$ is the set 
$$\mathcal{C}(T_{i})=\{T_{j}\in \mathcal{T}: \exists\, x,y\in T_{i}\cap T_{j} \text{ such that } x\neq y\}.$$
The \textit{(first) corona atlas }is the set of all (first) coronas that occur in $\mathcal{T}$, and a \textit{reduced (first) corona atlas} of $\mathcal{T}$ is a subset of the corona atlas of $\mathcal{T}$ that covers $\mathcal{T}$.
Similarly, the \textit{(first) vertex star} of a vertex $v$ is the set $\mathcal{V}(v)=\{T_{j}\in\mathcal{T}\mid T_{j}\cap \{v\} \neq\emptyset\}$, and a \textit{(first) vertex star atlas} is the set of all (first) vertex stars that occur in $\mathcal{T}$.
Finally, a tiling $\mathcal{T}$ is \textit{non-periodic }if it does not have translational symmetry in more than one direction, and a set of prototiles $\mathcal{S}$ is \textit{aperiodic} if it admits only non-periodic tilings.

\section{Penrose Tilings}

\begin{figure}\center{\includegraphics[scale=.3]{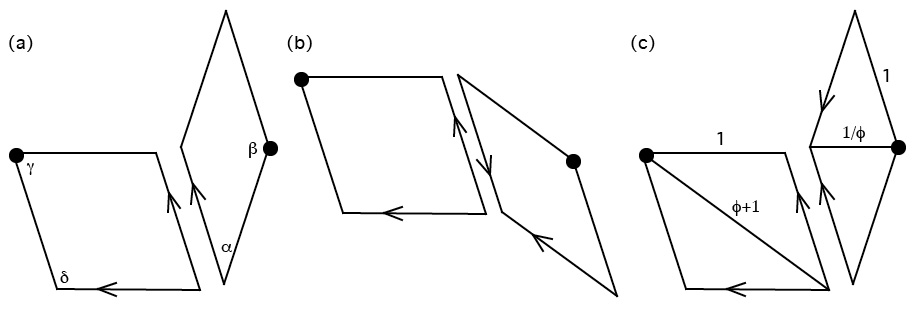}\caption{(a) the Penrose rhombs properly assembled, (b) the Penrose rhombs improperly assembled, (c) the Penrose rhombs split into triangles. }\label{rhombs}}\end{figure}

There are several types of tilings known as Penrose tilings. The type relevant for this paper is built from the two rhombic prototiles shown in Figure~\ref{rhombs}(a).
The sides of the rhombs are all of length one and the angles measure $\alpha=\frac{\pi}{5}, \beta=\frac{4\pi}{5},\gamma=\frac{2\pi}{5},$ and $\delta=\frac{3\pi}{5}$. In order to guarantee a non-periodic tiling, the edge and angle markings must line up with each other as in Figure~\ref{rhombs}(a).
Illegal configurations, such as the one in Figure~\ref{rhombs}(b),
either produce a periodic tiling of the plane or prevent a tiling of the plane.


\begin{theorem} 
The Penrose protoset, together with the adjacency rules, admits uncountably many non-congruent tilings of the plane, all of which are non-periodic (\cite{Senechal}, p189).
\end{theorem}

\begin{theorem}[Local Isomorphism] Every patch in a given Penrose tiling by rhombs occurs infinitely many times in every other Penrose tiling by rhombs, \textit{i.e.,} all Penrose tilings by rhombs are locally isomorphic (\cite{Senechal}, p175).
\end{theorem}

The rhombic prototiles may equivalently be thought of as pairs of triangles, as shown in Figure~\ref{rhombs}(c), the smaller with side lengths 1 and $\frac{1}{\phi}$, where $\phi$ is the golden ratio $\phi=\frac{1+\sqrt{5}}{2}=1.618\dots$, and the larger with side lengths 1 and $\phi+1$ (see Figure~\ref{rhombs}).

\begin{figure}\center{\includegraphics[scale=.3]{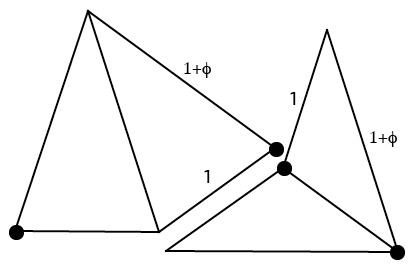}}\caption{The Penrose ``kite and dart" tiles.}\label{k&d}\end{figure}

\begin{remark}
Another type of Penrose tiling is one whose prototiles are commonly referred to as a ``kite" and a ``dart."  In this version, the acute triangle has side lengths measuring 1 and $\phi+1$ as in Figure~\ref{k&d}.
As in the case of rhombs, in order to produce a nonperiodic tiling the triangles of the kites and darts are assembled so that the marked vertices shown in Figure~\ref{k&d} coincide. We will be chiefly concerned with Penrose tilings by kites and darts only with regard to Penrose deflation.
\end{remark}

\begin{figure}\center{\includegraphics[scale=.25]{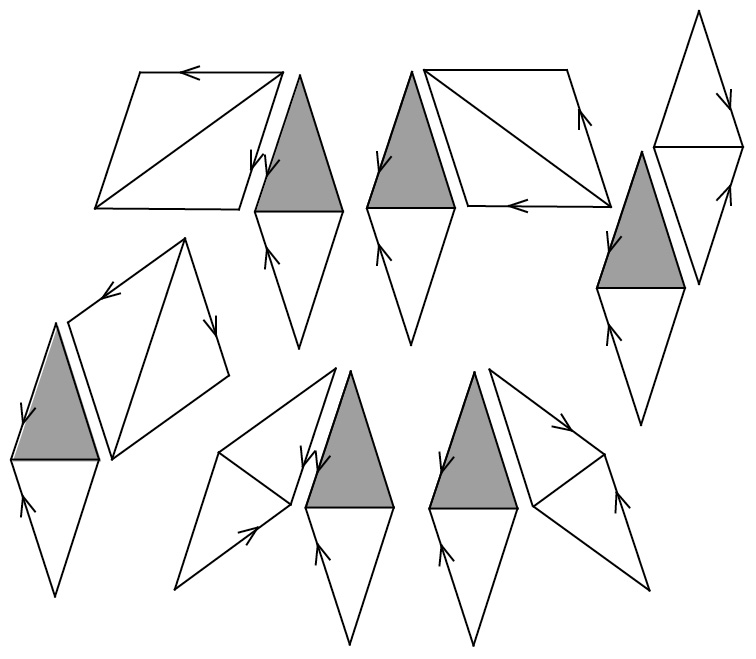}}\caption{The six locally legal placements of tiles adjacent to a small tile.}\label{def}\end{figure}
\begin{figure}\center{\includegraphics[scale=.5]{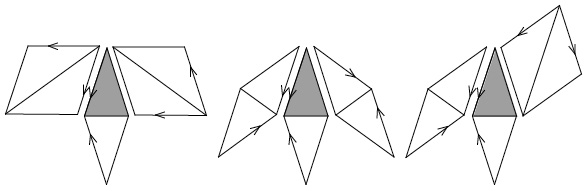}}\caption{Problematic configurations resulting from the locally legal placements of Figure \ref{def}.}\label{defbad}\end{figure}

\begin{proposition} Given a Penrose tiling by either rhombs or kites and darts, for each small triangle (half tile with divisions shown in Figures \ref{rhombs} and \ref{k&d}) there is exactly one large triangle adjacent to it such that the edge between them may be erased producing an even larger triangular tile similar to the original small triangle. 
\end{proposition}
\begin{proof}
Figure \ref{def} shows the six locally legal configurations of tiles adjacent to a small triangle in the rhomb case. Of these, all are globally legal except for the one on the top right (see Figure~\ref{pvstars}). The others may be combined into the possible problematic configurations shown in Figure~\ref{defbad}. In the leftmost configuration in Figure~\ref{defbad}, there are two tiles that might be combined with the center tile, and in the other two configurations there are none. However, these configurations are not globally legal (see Figure~\ref{pvstars}, (\cite{Senechal}, p177)). The remaining globally legal configurations satisfy the theorem.
 The kite and dart case is similarly easy to verify. \end{proof}

\begin{definition}\label{pdef}[Penrose Composition] Penrose composition is the process by which each small triangle in a Penrose tiling is amalgamated with an adjacent large triangle.  The adjacency rules ensure that each small triangle has exactly one such large triangle adjacent to it. When each small triangle is amalgamated with a large triangle in this way, the resulting tile is geometrically similar to the original small tile, and the tiling produced is a Penrose tiling.
\end{definition}

We will distinguish between Penrose tilings by constructing an identifying index sequence.  
\begin{algorithm} [Penrose Index Sequence] 
Given a Penrose tiling by triangles, label the tiles $s$ or $l$ depending on whether they are small or large. 
\begin{enumerate}
\item Pick an arbitrary point $P$ interior to a tile. 
\item If $P$ lies in a small tile, record an $s$. If it lies in a large tile, record an $l$. 
\item Perform the Penrose composition process \ref{pdef} on the tiling, thus eliminating all original small triangles from the tiling. A new Penrose tiling will result in which the original large tiles are the new small tiles. 
\item Return to step 2. 
\end{enumerate}
Notice that the tiling alternates between a tiling by rhombs and a tiling by kites and darts.
This process may be repeated indefinitely, producing an infinite index sequence for the tiling relative to $P$. 
\end{algorithm}

To identify a tiling independent of the choice of $P$, we define an equivalence relation $\sim$ on the set $  X_{p} $ of index sequences of Penrose tilings. (Note that $ X_{p} $ is the set of all sequences of $s$'s and $l$'s in which an $s$ is always followed by an $l$.) Let 
$$ \{x_{n}\} \sim \{y_{n}\} \Leftrightarrow \exists\, m \text{ such that } x_{n} = y_{n} \,\forall n \geq m$$ 
that is, two sequences are in the same equivalence class if they eventually coincide. This yields the quotient set $ X_{p} / \sim $ representing the set of all Penrose tilings.

\section{Amman Tilings and Their Combinatorial and Geometric Properties}

\begin{figure}\center{\includegraphics[scale=.5]{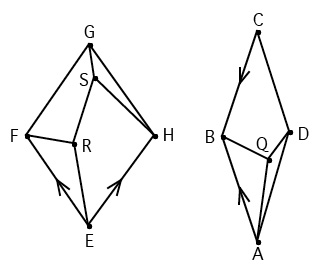}}\caption{The recomposition of Penrose thick and thin rhombs into Ammann tiles.}\label{recomp}\end{figure}
\begin{figure}\center{\includegraphics[scale=.5]{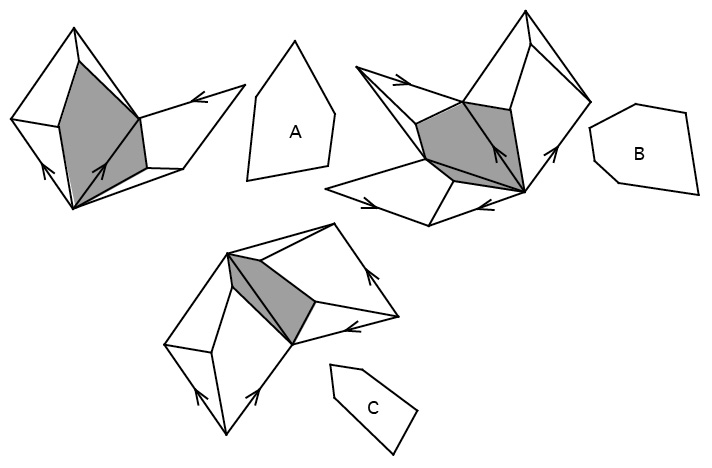}}\caption{The Ammann tiles created from recomposition.}\label{abc}\end{figure}

Amman tilings are derived from Penrose tilings by the process of \textit{recomposition} (\cite{GS}, p548) .

For the following algorithm we assume the orientation shown in Figure~\ref{recomp}.

\begin{algorithm}[Recomposition]\label{recompa}
Ammann's construction creates an Ammann tiling from a Penrose tiling by rhombs.
\begin{enumerate} 
\item Choose a point $Q$ within a single thin rhomb. Without loss of generality, we assume $Q$ is in the lower half of the rhomb.
\item Connect $Q$ to the three closest vertices of the rhomb. 
\item Copy this construction into all thin rhombs.
\item Copy $\triangle ABQ$ into the lower left of each thick rhomb such that $\overrightarrow{AB}\mapsto\overrightarrow{EF}$ and a new point is created $Q\mapsto R$ to yield $\triangle EFR$ within the thick rhomb.
\item Copy $\triangle DAQ$  into the upper right of each thick rhomb such that $\overrightarrow{DA}\mapsto\overrightarrow{GH}$ and a new point is created $Q\mapsto S$ to yield $\triangle GHS$ within the thick rhomb.
\item Connect points $R$ and $S$ within every thick rhomb.
\item Copy $\triangle GHS$ and $\triangle GHS$ to all of the thick rhombs.
\item Erase the edges of the original Penrose tiling.
\end{enumerate}
\end{algorithm}

This construction uniquely determines a tiling once $Q$ is chosen. 
In order to  ensure a non-periodic tiling of the plane, $Q$ must be chosen so that no two of $\abs{AQ},\,\abs{BQ},\,\abs{CQ},\,\abs{RS}$ are equal. If any two of them are equal, the three resulting Ammann prototiles may admit some periodic tilings of the plane.

\begin{figure}\center{\includegraphics[scale=.5]{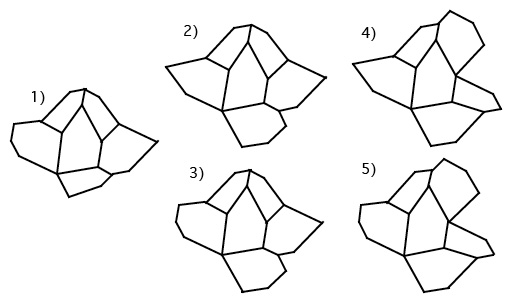}}\caption{The five coronas of an Ammann $A$ tile.}\label{acor}
\end{figure}

\begin{theorem}
 Given an Ammann tiling $\mathcal{T}$ constructed from a Penrose tiling $\mathcal P$, the five coronas illustrated in Figure~\ref{acor} are the only possible coronas of an $A$ tile of $\mathcal T$.
\end{theorem}

\begin{proof}
 (Sketch) As can be seen by inspection of Figure~\ref{abc}, there is a bijective correspondence between thick rhombs in $\mathcal P$ and type $A$ tiles in $\mathcal T$. So, to determine the possible coronas of $A$, we consider the possible Penrose coronas of a thick rhomb. The possible coronas of a Penrose thick rhomb are determined by the Penrose vertex atlas, and by conducting the recomposition algorithm on the tiles of these Penrose coronas it can be seen that these five are the only possible coronas of an Ammann type A tile.
\end{proof}

\begin{theorem}
The set of Ammann coronas of $A$ is a reduced corona atlas.
\end{theorem}

\begin{proof}
   As in the previous theorem, let $\mathcal P$ be the Penrose tiling corresponding to an Ammann tiling $\mathcal T$. Assume for contradiction that the set of coronas of $A$ does not cover $\mathcal T$. Then, there is at least one corona of a $B$ or $C$ tile that does not contain any $A$ tiles. 
   
 \begin{figure}\center{\includegraphics [scale=.5]{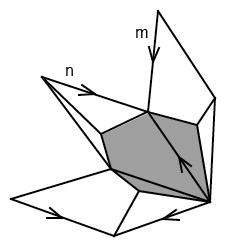}}\caption{A locally legal construction of a $B$ tile that is not globally legal.}\label{illegal}\end{figure}
 
   By inspection of Figure~\ref{abc}, we can see that a thick rhomb is needed to create a $C$ tile, so every $C$ tile has an $A$ tile in its corona. On the other hand, the patch shown in Figure~\ref{illegal} shows that it it is locally legal to assemble three thin rhombs with Ammann markings to form a $B$ tile. However, this configuration is not globally legal, as is immediately apparent when we try to put another tile between lines $m$ and $n$. So, this configuration is impossible, and this set of five coronas of type $A$ tiles covers $\mathcal T$, and is thus a reduced corona atlas.
\end{proof}

\begin{figure}\center{\includegraphics[scale=.5] {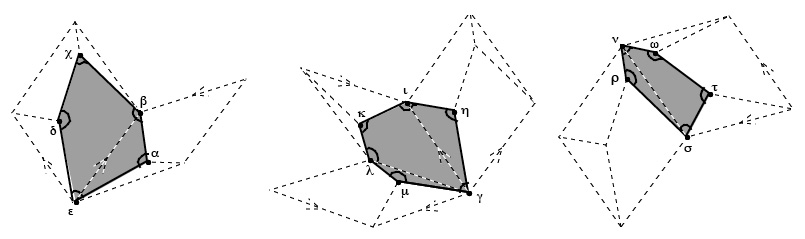}}\caption{}\label{abcangles}\end{figure}

\begin{proposition} \label{cangles} Let $\theta=\frac{\pi}{5}$. 
Based on the labeling in Figure~\ref{abcangles}, the following angle relations hold for Ammann tiles.
\begin{enumerate}
\item $\varepsilon=\gamma=\nu=2\theta $
\item
$\iota=\lambda=4\theta$ \item
$\beta+\sigma=6\theta$ \item
$\chi+\rho+\omega=10\theta$ \item
$\delta+\tau+\eta=10\theta$ \item
$\alpha+\kappa+\mu=10\theta $\item
$\alpha=\eta $\item
$\mu=\rho$.
\end{enumerate}
\end{proposition}

\begin{figure}\center{\includegraphics [scale=.5]{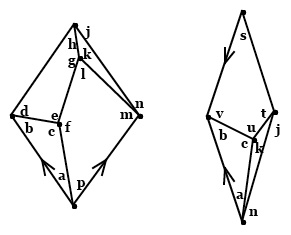}}\caption{}\label{recomp2}\end{figure}

\begin{proof}
Since the thick rhombs have angles of $\frac{2}{5}\pi$ and $\frac{3}{5}\pi$, and the thin rhombs have angles of $\frac{1}{5}\pi$ and $\frac{4}{5}\pi$, we see in Figure~\ref{recomp2} that
\begin{enumerate}
\item $a+p=h+j=\frac{2}{5}\pi=2\theta$
\item$ b+d=n+m=\frac{3}{5}\pi=3\theta$
\item $a+n=s=\frac{1}{5}\pi=\theta$
\item $b+v=t+j=\frac{4}{5}\pi=4\theta.$
\end{enumerate}

Furthermore, we can see in Figure~\ref{abcangles} that
$\alpha= c, \quad 
 \beta=b+m, \quad  
 \chi=l   ,\quad 
 \delta=f ,\quad 
\\ \varepsilon=p+a=
 \gamma=a+s+n,\quad 
\eta=    c,\quad 
 \iota=b+v,\quad 
 \kappa=u,\quad 
 \lambda=t+j ,\quad 
 \mu=k,\quad 
\\ \nu=h+j,\quad 
 \rho=k,\quad 
 \sigma=d+n,\quad 
 \tau=e,\quad 
 \omega=g.$
It is easy to verify that these relations give us the desired result.
\end{proof}

\begin{figure}\center{\includegraphics[width=3in] {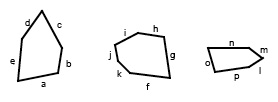}}\caption{}\label{abcedges}
\end{figure}

\begin{proposition}\label{cedges}
Referring to the labeling in Figure~\ref{abcedges}, the following edge congruences hold.
\begin{enumerate} \item $a=c=e=f=g=n $ \item$ b=h=i=o$\item $j=k=l=m$  \item $d=p$\end{enumerate}
\end{proposition}

\begin{proof}
The proposition is evident by inspection of Figures~\ref{abcedges} and \ref{abcangles}.
\end{proof}

\begin{theorem}\label{abcaperiodic}
Let $\mathcal P$ be a Penrose tilng and let $\mathcal T$ be a tiling obtained via the recomposition process in Algorithm \ref{recompa}. Let $\{A, B, C\}$ denote the protoset of $\mathcal T$ as labeled in Figure \ref{abc}. Then if $\mathcal T'$ is any other tiling admitted by $\{A, B, C\}$ then there exists a Penrose tiling $\mathcal P'$ such that $\mathcal T'$ is obtained from $\mathcal P'$ by recomposition.
\end{theorem}

\begin{proof}
Because the Ammann vertex atlas determines all Ammann tilings and all Ammann tilings derived from Penrose tilings are non-periodic by construction, it is sufficient to show that all Ammann vertex stars are derived from globally legal Penrose patches.

\begin{figure}\center{\includegraphics[scale=.5] {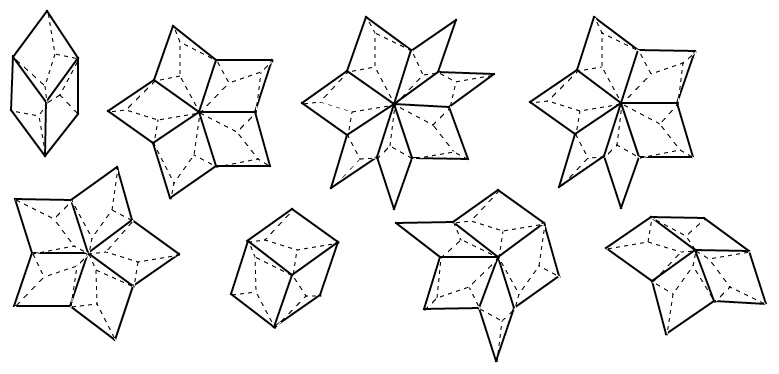}}\caption{The eight globally legal Penrose vertex stars (\cite{Senechal}, p177).}\label{pvstars}\end{figure}

\begin{figure}\center{\includegraphics[scale=.5] {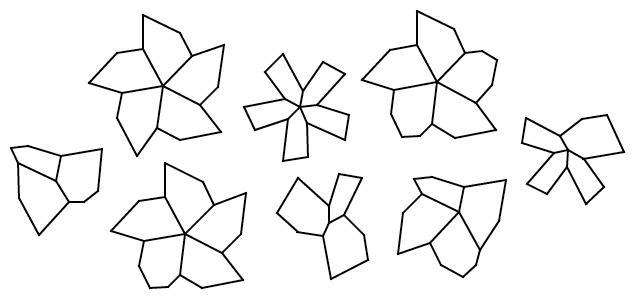}}\caption{The eight Ammann vertex stars derived from the eight Penrose vertex stars in Figure~\ref{pvstars}.}\label{avs1}\end{figure}

From the Penrose vertex atlas shown in Figure~\ref{pvstars}, we get the eight Ammann vertex stars of $\mathcal T$ shown in Figure~\ref{avs1}. Since these were constructed from the Penrose vertex star atlas, they are globally legal.

\begin{figure}\center{\includegraphics[scale=.4] {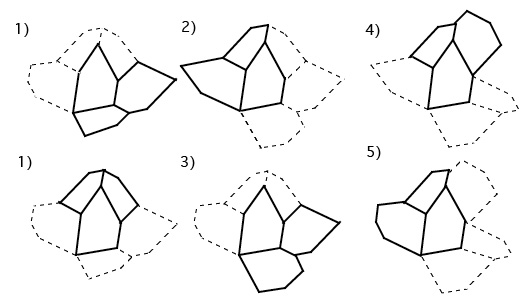}}\caption{Six globally legal vertex stars obtained from the five coronas of a type $A$ tile.}\label{acoronas2}\end{figure}

The recomposition process added the three vertices $Q$, $R$, and $S$, so we now consider their possible vertex stars. The five coronas of $A$ give us six more globally legal vertex stars shown in Figure~\ref{acoronas2}. These were also constructed from the recomposition of a Penrose tiling, so they are globally legal as well. This gives us fourteen globally legal vertex stars of $\mathcal T$.

\begin{figure}\center{\includegraphics[scale=.5] {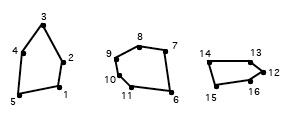}}\caption{}\label{abcnumbers}\end{figure}
\begin{figure}\center{\includegraphics [scale=.4]{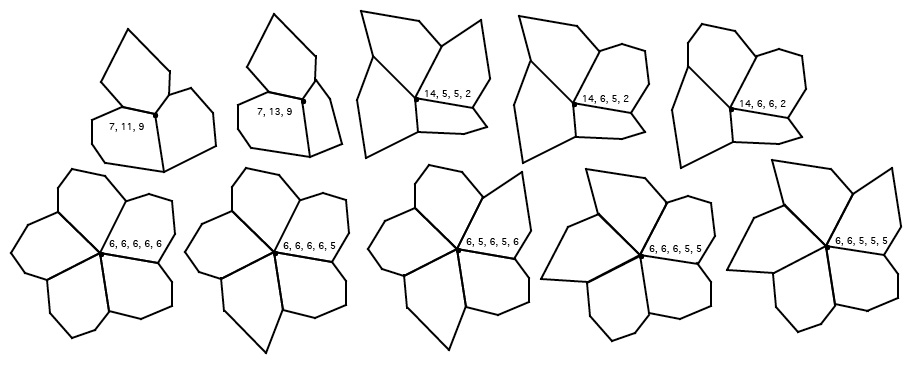}}\caption{}\label{tenillegals}\end{figure}
\begin{figure}\center{\includegraphics [scale=.4]{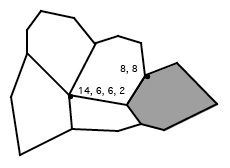}}\caption{}\label{8,8}\end{figure}

Next, we use the angle and edge relations established in Propositions \ref{cangles} and \ref{cedges} to check if there are any other locally legal Amman vertex stars, and we find that there are only the ten shown in Figure~\ref{tenillegals}.
 (This is where we use the condition that the segments constructed in Algorithm \ref{recompa} are of different lengths. If any were instead the same length, there would be more locally legal vertex stars than those shown here.) The central vertex of each vertex star in Figure~\ref{tenillegals} is labeled with the numbers of the vertices that meet there as per the numbering in Figure~\ref{abcnumbers}.

\begin{claim} Each of the ten vertex configurations in Figure~\ref{tenillegals} is not globally legal.\end{claim}

\begin{figure}\center{\includegraphics [scale=.5]{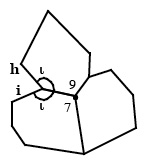}}\caption{No tile has two adjacent sides of length $i=h$ and angle between them of $10\theta-2\iota=2\theta$, so the empty triangle on the left can never be filled.}\label{9+7=!}\end{figure}

Notice that (7,13,9) and (7,11,9) are not globally legal because whenever vertices 7 and 9 meet, a space is created that cannot be filled (see Figure~\ref{9+7=!}).
Furthermore, when two instances of vertex 6 meet, vertices 7 and 11 also meet. But if vertices 7 and 11 meet, then vertex 9 meets there as well, because $\eta+\mu=10\theta-\kappa$. However, as stated above, (7,11,9) is not globally legal, so a vertex star where two instances of vertex 6 meet is not globally legal.
Therefore, (6,6,6,6,6), (6,6,6,6,5), (6,6,6,5,5), (6,6,5,5,5), (6,5,6,5,6), and (2,14,6,6) are not globally legal.

Now, we focus our attention to vertex stars (14,6,5,2) and (14,6,6,2). In both cases, only a $B$ tile can fit adjacent to the $C$ and $B$ tiles on the right side of the vertex stars (see Figure \ref{8,8}) because the angle between them is $10\theta-\nu-\rho=\kappa$ and the edge lengths are $m$ and $h$. This creates vertex (8,8) in which a tile would fit with a vertex angle of $2\theta$ between congruent edges of length $h=i=b=o$. Since no such tile exists, the vertex star (8,8) cannot be completed, so vertex stars (14,6,5,2) and (14,6,6,2) are not globally legal.
Finally, the left side of vertex star (14,5,5,2) contains vertex (3,4). Two edges of length $d$ meet there at an angle of $10\theta-\chi-\delta$. Since no tile fits in this space, the vertex star is not globally legal.
 Therefore, each of the ten vertex stars in Figure \ref{tenillegals} is not globally legal, and the vertex star atlas of $\mathcal T$ contains only the aforementioned fourteen vertex stars, which are derived from $\mathcal P$. This proves the claim.

The Penrose vertex star atlas of eight vertex stars completely determines all Penrose tilings (\cite{Senechal}, p177). So, the set of fourteen vertex stars of $\mathcal T$ that are derived from $\mathcal P$ completely determines all  tilings that can be derived from a Penrose tiling by the recomposition process in Algorithm \ref{recompa}. But the set of vertex stars of $\mathcal T$ is identical to the vertex star atlas of $\mathcal T'$. Therefore, an arbitrary Ammann tiling can be derived from a corresponding Penrose tiling. 
\end{proof}

\begin{corollary}
The protoset $\{A,B,C\}$ is an aperiodic protoset.
\end{corollary}

\begin{proof}
Since Penrose tilings are non-periodic, by the previous theorem so too are $\mathcal T$ and $\mathcal T'$. Therefore, A, B, and C form an aperiodic protoset.
\end{proof}

\begin{remark}
The protoset $\{A,B,C\}$ has no adjacency rules, unlike the underlying Penrose rhombs. These results motivate the following definition.
\end{remark}

\begin{definition} \label{Adef} Using the edge labeling of Figure~\ref{abcedges} and the angle labeling of Figure~\ref{abcangles}, we define an Ammann tiling to be a tiling of the plane admitted by a protoset of three tiles, two pentagons and a hexagon, satisfying the edge relations
of Proposition \ref{cedges}
and the angle relations of Proposition \ref{cangles}
\end{definition}

\begin{remark}
The angle and edge conditions in the previous definition are equivalent to specifying the vertex star atlas of Ammann types.
\end{remark}

\begin{remark}
There are several types of tilings already referred to as Ammann tilings in the literature. Therefore, one should consider Definition~\ref{Adef} to be local to this paper.
\end{remark}

\section{Iterating Ammann Tilings}

\begin{figure}
\center{\includegraphics[scale=.5]{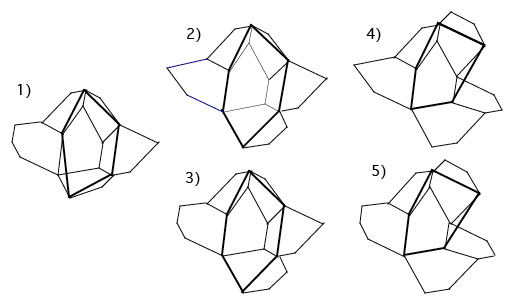}

\includegraphics[scale=.5]{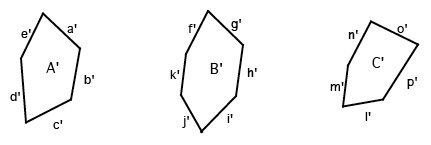}\caption{Algorithm \ref{ait} creates the three tiles of $\mathcal T'$ from the coronas of $A$ in $\mathcal T$. Note that if the tiles were rescaled by a factor of $\frac{1}{\phi}$ the new tiles would be similar in size to the originals.}\label{iter}}
\end{figure}

\begin{algorithm}[Ammann Iteration]\label{ait}
   Let $\mathcal T$ be an Ammann tiling.  By connecting vertices within the five coronas of $A$ as shown in Figure~\ref{iter} and then erasing the original edges, we create a new tiling of the plane. Examining all possible arrangements of the five coronas, we see that this process produces a tiling by the three prototiles shown in Figure~\ref{iter}. We label the new tiling $\mathcal T'$ and call the new tiles from corona 1 the type $A'$ tiles of $\mathcal T'$, the tiles from coronas 2 and 3 we call the type $B'$ tiles, and the tiles from corona 4 and 5 we call the type $C'$ tiles. The angles and edges of the tiles of $\mathcal T'$ are labeled in reverse, \textit{e.g.} if the tiles of $\mathcal T$ are labeled counter-clockwise, then those of $\mathcal T'$ are labeled clockwise.
    \end{algorithm}

\begin{theorem}
     Let $\mathcal T$ be an Ammann tiling.
     The tiling $\mathcal T'$ obtained via Algorithm \ref{ait} is also an Ammann tiling. 
\end{theorem}

\begin{proof}

By construction, as shown in Figure~\ref{iter} the tiles in $\mathcal T'$ obey 
      \begin{enumerate} \item $a'=c'=e'=f'=g'=n'$\item  $b'=h'=i'=o'$\item  $k'=m'$. 
      \end{enumerate}

Since in $\mathcal T$, $a=e=g$, we have $j'=k'=m'=l'$. Also, it can be easily shown that only an $A$ tile would fit between the $C$ and $B$ tiles on the far right in coronaa 5 in Figure~\ref{135}.  This yields the edge length equality $d'=p'$.
Therefore, we get the the fourth relation in Proposition \ref{cedges} required of the edges in an Ammann protoset. 
So, the algorithm preserves the edge congruence relations.

By examining Figure~\ref{iter}, it is clear that the angles of $\mathcal T$ are not congruent to the angles of $\mathcal T'$. However, we aim to show that the angle restrictions imposed by our algorithm for constructing $\mathcal T'$ are the same as the restrictions present in $\mathcal T$.

\begin{figure}\center{\includegraphics [scale=.5]{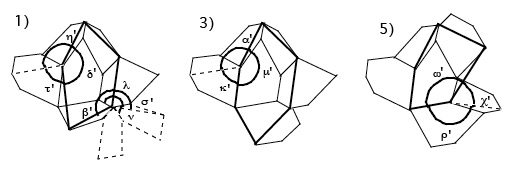}\caption{}\label{135}}
\end{figure}

From Figure \ref{135} we get immediately:
\begin{enumerate}
\item $ \nu'= \varepsilon =2\theta$,
\item $\varepsilon' = \gamma'=\nu=2\theta$,
\item $\iota' =\lambda=4\theta$,
\item $\lambda'=\varepsilon+\gamma=4\theta$,
\item $\mu'= \rho' $, and
\item $\alpha'=\eta' $
\end{enumerate}

The circle in the upper left of corona 1 in Figure~\ref{135} shows that $\delta' +\eta' +\tau' =10\theta$.
The circle on the upper left of corona 3 shows that  $\alpha'+\kappa'+\mu'=10\theta$.
The circle in the lower right of corona 5 shows that $\chi'+\rho' +\omega' =10\theta$.
Now we turn our attention to the lower right of corona 1.  The larger of the two arcs marks angle $\beta'$ and the smaller marks angle $\sigma'$. Recall from the labeling in Figure~\ref{abcangles} that $\lambda=4\theta$ and $\nu=2\theta$. 
From these we get $\beta' +\sigma'=\lambda+\nu=6\theta$.

This gives us angle restrictions
\begin{enumerate}
\item $\varepsilon=\gamma' =\nu'=2\theta$ \item
$\iota'=\lambda'=4\theta $\item
$\beta'+\sigma'=6\theta$ \item
$\gamma'+\rho'+\omega'=10\theta$ \item
$\delta'+\tau'+\eta'=10\theta $\item
$\alpha'+\kappa'+\mu'=10\theta $\item
$\alpha'=\eta' $\item
$\mu' =\rho'. $
\end{enumerate}
These are exactly the  same angle relations of Proposition \ref{cangles} required of an Ammann tiling. Therefore, 
$\mathcal T'$ is an Ammann tiling in the sense of Definition \ref{Adef}.
\end{proof}

Recall that Penrose composition done twice on a Penrose tiling by rhombs produces another Penrose tiling by rhombs whose prototiles are similar to the originals but scaled by a factor of $\phi$. We define the notation 
\begin{equation*}
\xymatrix {
& \mathcal P\ar[r]^{\text{Comp.$^2$}} &\mathcal P'}
\end{equation*}
 for Penrose tilings $\mathcal P$ and $\mathcal P'$ by rhombs to refer to performing Penrose compostion twice. Analogously, we define the notation 
 \begin{equation*}
\xymatrix {
& \mathcal T\ar[r]^{\text{Iter.}} &\mathcal T'}
\end{equation*}
for Ammann tilings $\mathcal T$ and $\mathcal T'$ to refer to Ammann iteration as defined in Algorithm \ref{ait}.

\begin{figure}\center{\includegraphics[scale=.5]{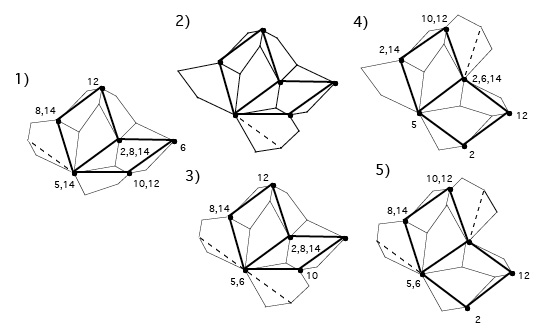}}\caption{The Ammann coronas with associated Penrose rhombs.}\label{corswithpen}
\end{figure}
\begin{figure}\center{\includegraphics[scale=.5]{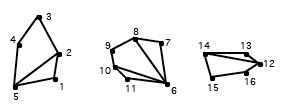}}\caption{Divisions of Ammann tiles by Penrose rhombs of corresponding Penrose tiling $\mathcal P$.}\label{cutatiles}\end{figure}

\begin{theorem} \label{R=P'}
Given an Ammann tiling $\mathcal T$, let $\mathcal P$ be its underlying Penrose tiling. Let $\mathcal T'$ be the iterated Ammann tiling, let $\mathcal P'$ be the tiling produced by applying Penrose composition to $\mathcal P$ twice, and let $\mathcal R$ be the underlying Penrose tiling of $\mathcal T'$. Then $\mathcal P'=\mathcal R$.
\begin{equation*}
\xymatrix {\text{Penrose Tilings:}
& \mathcal P\ar[r]^{\text{Comp.$^2$}} &\mathcal P'\ar@(r,r)[dd]\\
\text{Ammann Tilings:}&\mathcal T\ar[u]\ar[r]^{\text{Iter.}} &\mathcal T'\ar[d]\\
\text{Penrose Tilings:}& ~&\mathcal R\ar@(r,r)[uu]
}
\end{equation*}
\end{theorem}

\begin{proof}

First, note that $\mathcal R$ is well defined by Theorem \ref{abcaperiodic}. (Although Ammann tilings were not specifically defined until after the proof of Theorem \ref{abcaperiodic}, the proof used only the edge and angle relations that were later used to define Ammann tilings, allowing us to use its result here.)

Next, consider the relationship between $\mathcal T$ and $\mathcal P$. The edges of $\mathcal P$ cut the tiles of $\mathcal T$ in exactly the same way for each type of Ammann tile, as shown in Figure~\ref{cutatiles} (each type $A$ tile of $\mathcal T$ is cut once by segment $\overline{2,5}$, $B$ tiles are cut twice by segments $\overline{8,6}$ and $\overline{10,6}$, and $C$ tiles are cut once by segment $\overline{12,14}$). Since $\mathcal T'$ is an Ammann tiling with corresponding Penrose tiling $\mathcal R$, $\mathcal R$ will cut the tiles of $\mathcal T'$ across the corresponding vertices.  To show that $\mathcal P'=\mathcal R$, it is sufficient to show that $\mathcal P'$ makes exactly the same divisions in the tiles of $\mathcal T'$ as $\mathcal R$.

\begin{figure}\center{\includegraphics[scale=.5]{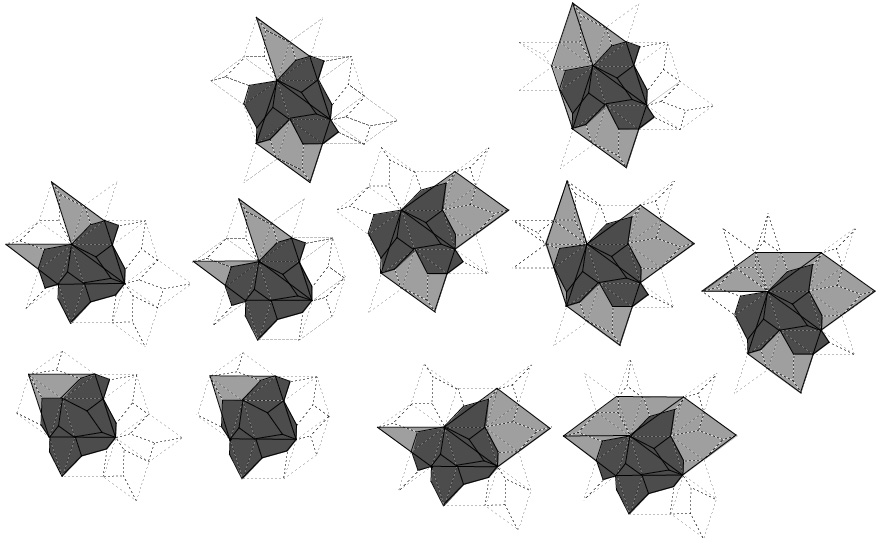}}\caption{The Ammann coronas with Penrose rhombs of $\mathcal P'$. }\label{T,P'}
\end{figure}

\begin{figure}\center{\includegraphics[scale=.5]{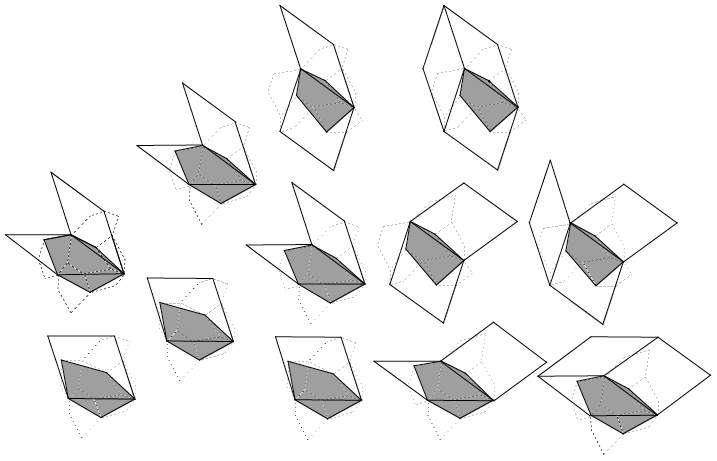}}\caption{The same coronas as Figure~\ref{T,P'} but with the tiles of $\mathcal T'$ drawn in. }\label{T,P'2}
\end{figure}

\begin{figure}\center{\includegraphics[scale=.5]{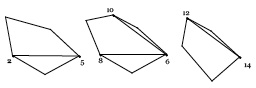}}\caption{The tiles of $\mathcal T'$ with the cuts made by $\mathcal P'$. }\label{T,P'3}
\end{figure}

Figure~\ref{T,P'} shows the five Ammann coronas of $\mathcal T$ in the context of all possible second coronas of the Ammann $A$ tiles. The Penrose rhombs of $\mathcal P'$ are superimposed. Iterating, we see in Figures~\ref{T,P'2} and \ref{T,P'3} that the tiles of $\mathcal T'$ are cut in same way by $\mathcal P'$ as they are by $\mathcal R$. Therefore, $\mathcal P'=\mathcal R$.
\end{proof}

Since $\mathcal T'$ is an Ammann tiling, the iteration algorithm may be performed on $\mathcal T'$ etc., yielding an infinite sequence of Ammann tilings corresponding to the infinite sequence of Penrose tilings created by the Penrose double composition process.

\section{Dynamics of the Ammann Iteration Process}

The Ammann iteration process does not preserve the exact shapes of the prototiles, making it difficult to  compare Ammann tilings obtained through repeated iteration.  On the other hand, Penrose tiles under double composition are easily described. Ammann iteration thus may be tracked by simultaneously tracking the change in the underlying Penrose tiling and the change in the relative location of $Q$ within the Penrose thin rhombs. Recall that the Penrose double composition process scales the prototiles by a factor of $\phi$, so by rescaling both the Penrose and Ammann tilings by $\frac{1}{\phi}$ after each Ammann iteration (equivalently Penrose double composition), the dimensions of the tiles of the underlying Penrose tiling remain constant throughout, allowing us to quantitatively compare the location of $Q$ within the thin rhomb at each stage of repeated iteration. 
We may thus study the dynamics of the iteration process on Ammann tilings by appealing to the local isomorphism theorem and tracking the movement of $Q$ in an Ammann tiling with respect to the (changing) underlying Penrose tiling.

Since each Ammann iteration reverses the direction of the labeling of the Ammann prototiles, the location of $Q_n$ will alternate between the lower right and lower left quadrants of the reference thin Penrose rhomb. In order to simplify our analysis, we let the sequence $\{Q_0, Q_1, Q_2, \dots, Q_n,\dots\}$ represent the movement of $Q$ under repeated iteration, but with the odd elements of the sequence reflectied through the principal (long) axis of the reference thin rhomb. We locate $Q$ within a  Penrose thin rhomb using polar  coordinates based along one edge of the rhomb.  Accordingly, we locate $Q_n$ by the parameters $r_n, \theta_n$.

\begin{figure}\center{\includegraphics[scale=.5]{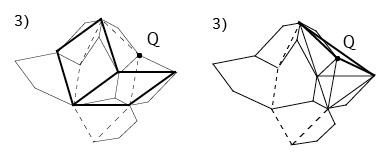}\caption{}\label{findQ2}}\end{figure}

Since we will first examine what happens to $Q$ under one application of the iteration process, we denote $Q':=Q_1$. We locate $Q'$ in corona 3 illustrated in Figure~\ref{findQ2} and notice that in this case, the point $Q$ is identical to the point $Q'$, but the orientation of the reference triangle has changed.

\begin{figure}\center{\includegraphics[scale=.5]{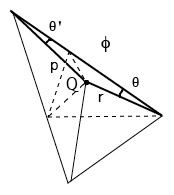}\caption{The enlarged triangle from corona 3), Figure~\ref{findQ2}. Notice that $Q$ is located within the dashed triangle by $(r,\theta)$ and within the solid triangle by $(p,\theta')$.}\label{findQ}}\end{figure}

Using the law of cosines on the triangle in Figure~\ref{findQ}, we get: 
$$p^2=r^2 +\phi^2-2r\phi\cos{\theta}.$$
 Normalizing, so that $\frac{p}{\phi}=r'$, we arrive at the formula
 \begin{equation}
 r'=\frac{\sqrt{r^2+\phi^2-2r\phi\cos\theta}}{\phi}.
 \label{r'}
 \end{equation}
 Again, we using the law of cosines on the triangle in Figure~\ref{findQ},

 \begin{equation}
 \begin{aligned}
 r^2&=\phi^2+p^2-2p\phi\cos{\theta'}
 \\& =\phi^2+(r')^2\phi^2-2\phi^2r'\cos{\theta'}.
\end{aligned}
\nonumber
  \end{equation}
  
 Rearranging we have
$$
\cos{ \theta'}={\left(\frac{\phi^2+(r')^2\phi^2-r^2}{2\phi^2r'}\right)},
$$
 and substituting for $r'$ using \ref{r'}, we have 
 
$$
\cos{\theta'}= \frac{\phi^2+(\frac{\sqrt{r^2+\phi^2-2r\phi\cos\theta}}{\phi})^2\phi^2-r^2}{2\phi^2\frac{\sqrt{r^2+\phi^2-2r\phi\cos\theta}}{\phi}},
$$
which simplifies to
\begin{equation}
\cos \theta'=\frac{\phi-r\cos\theta}{\sqrt{r^2+\phi^2-2r\phi\cos\theta}}. 
\end{equation}

Thus, the movement of $Q$ is described by the assignment $(r,\theta)\mapsto(r',\theta')$ where
\begin{equation}
(r',\theta')=\left(\frac{\sqrt{r^2+\phi^2-2r\phi\cos\theta}}{\phi}, \quad\arccos{\left(  \frac{\phi-r\cos\theta}{\sqrt{r^2+\phi^2-2r\phi\cos\theta}}   \right)} \right)\label{newQ}.
\end{equation}

\begin{theorem}
The map $(r,\theta)\mapsto(r',\theta')$ in \ref{newQ}  has a unique attractive fixed point at $(r,\theta)=(\frac{1}{\phi},0)$.\label{limit}
\end{theorem}

\begin{proof}
Changing to cartesian coordinates,
$$
x'=r'\cos{\theta'}
=\left( \frac{\sqrt{x^2+y^2+\phi^2-2x\phi}}{\phi}\right)
{\left(  \frac{\phi-x}{\sqrt{x^2+y^2+\phi^2-2x\phi}}   \right)}
=1-\frac{x}{\phi}
$$

and 
$$\begin{aligned}
y'&= r'\sin{\theta'}
\\&=\frac{\sqrt{x^2+y^2+\phi^2-2x\phi}}{\phi}\sin\arccos {\left(  \frac{\phi-x}{\sqrt{x^2+y^2+\phi^2-2x\phi}}   \right)}. 
\end{aligned}$$
Using the identity $\sin\arccos {x}=\sqrt{1-x^2}$, this simplifies to

$$\begin{aligned}
y'&=\frac{\sqrt{x^2+y^2+\phi^2-2x\phi}}{\phi}\left(\sqrt{1-\frac{(\phi^2-x\phi)^2}{\phi^2(x^2+y^2+\phi^2-2x\phi)}}\right)
\\&=\frac{\sqrt{\phi^2(x^2+y^2+\phi^2-2x\phi)-(\phi^4-2x\phi^3+x^2\phi^2)}}{\phi^2}
\\&=\frac{\abs{y}}{\phi}.
\end{aligned}$$
Since we are only interested in positive values of $y$ since $0\leq \theta \leq \frac{\pi}{5}$, we simply write 
\begin{equation}
y'=\frac{y}{\phi}.
\end{equation}
One more change of variables transforms this affine map to a linear map. Let
$$ u=x-\frac{1}{\phi}\text{ and }v=y.$$
Then 
\begin{equation}
u'=-\frac{u}{\phi}\text{ and } v'=\frac{v}{\phi}.
\label{map}
\end{equation}
This is a linear map which fixes $(u,v)=(0,0)$ and no other point. Going back to the previous coordinate system we have that the map fixes $(x,y)=(\frac{1}{\phi},0)$, which is equivalent to the point $(r,\theta)=(\frac{1}{\phi},0)$. 
Furthermore, since $0<\frac{1}{\phi}<1$, the eigenvalues for (\ref{map}) have absolute value less than 1, so the fixed point is attractive.
\end{proof}

\begin{remark}
The result of Theorem \ref{R=P'} allows us to identify an Ammann tiling by its underlying Penrose tiling along with the location of the point $Q$ within the thin Penrose rhombs. While Theorem \ref{limit} proves that $Q$ approaches a limit as the iteration process is repeated, it does not prove that the entire tiling approaches a limit. Considering a sequence of 0s and 1s that identifies a Penrose tiling, Penrose composition is equivalent to performing the shift map on that sequence. This map is chaotic, and except in a limited number of special cases, the underlying Penrose tiling does not approach a limit.
\end{remark}

\section{Diffraction Properties and Possible Relations to Quasicrystals}
\begin{figure}\center{\includegraphics[scale=.5]{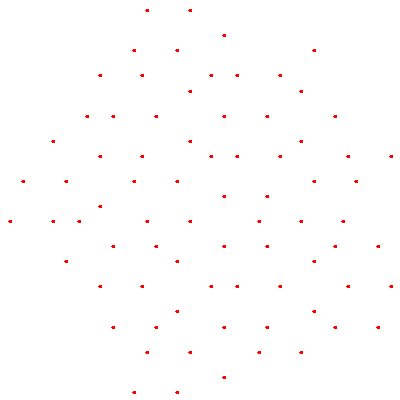}\includegraphics[scale=1]{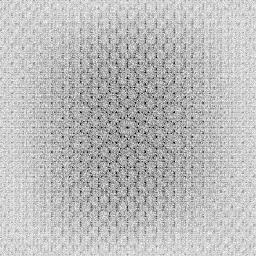}\caption{Left: the vertices of a patch of a Penrose tiling by rhombs. Right: its diffraction pattern.}\label{p}}
\end{figure}
\begin{figure}\center{\includegraphics[scale=.5]{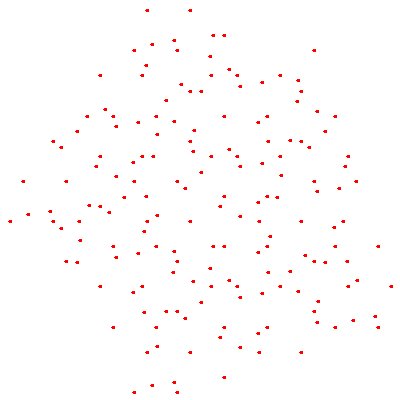}\includegraphics[scale=1]{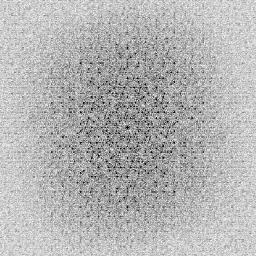}\caption{Left: the vertices of a patch of a generic Ammann tiling. Right: its diffraction pattern.}\label{ga}}
\end{figure}
\begin{figure}\center{\includegraphics[scale=.5]{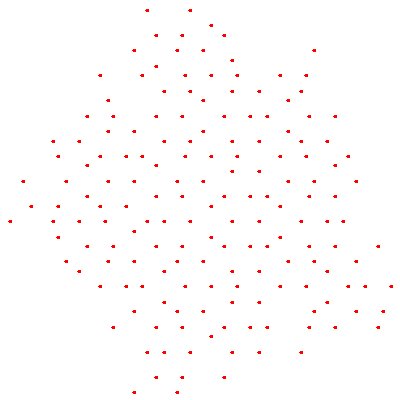}\includegraphics[scale=1]{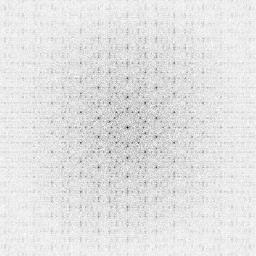}\caption{Left: the vertices of a patch of an Ammann limit tiling. Right: its diffraction pattern.}\label{a}}
\end{figure}

It is well known that a three-dimensional version of the Penrose tiling by rhombs is used to model quasicrystals with five-fold symmetry, so it is natural to ask whether Ammann tilings are similarly useful. 
The symmetry of a quasicrystal is identified by examining its Fraunhofer diffraction pattern [the far-range x-ray diffraction pattern]. This pattern can be simulated for tilings using the Fourier transform. We describe the set of vertices of the tiling [molecules of the quasicrystal] by the generalized function that has mass one at each point of the set and is zero everywhere else. The diffraction pattern of the vertex set of the tiling is then given by the intensity plot of the magnitude squared of the Fourier transform of this generalized function.

Figures \ref{p}, \ref{ga}, and \ref{a} each show the vertices of a tiling represented by small circles (left) and the inverted image of that set's diffraction pattern (right). Although the generic Ammann tiling (Figure \ref{ga}) does not show very clear symmetry, the Ammann limit tiling (Figure \ref{a}) shows very pronounced peaks, which are even brighter than those of the corresponding Penrose tiling (Figure  \ref{p}). This may indicate that the Ammann limit tiling also models a quasicrystal.



\begin{thebibliography} {99}
\bibitem{deB} deBruijn, N.G. ``Algebraic Theory of Penrose's Non-periodic Tilings of the Plane 1." Proceedings of the Koninklijke Nederlandse Akademie Van Wetenschappen Series A-mathematical Sciences, 1981.
\bibitem{GS}Gr\"unbaum, Branko and G.C. Shephard. \textit{Tilings and Patterns}. W.H. Freeman and Company, New York, 1987.
\bibitem{LP} Lunnon, W.F. and P.A.B. Pleasants. ``Quasicrystalographic tilings."  Journal De Mathematiques Pure et Appliquees, Vol 66. Cauthier-Villars, 1987.
\bibitem{Paterson} Paterson, Alan L.T. \textit{Groupoids, Inverse Semigroups, and Their Operator Algebras}. Birkh\"auser, Boston, 1999.
\bibitem{Senechal} Senechal, Marjorie. \textit{Quasicrystals and Geometry}. Cambridge University Press, New York, 1995.
\end{thebibliography}
\end{document}